\documentclass[12pt,reqno]{amsart}
\usepackage{amssymb}
\usepackage{mathtools}
\usepackage{txfonts}
\usepackage{eucal}
\usepackage{extarrows}

\usepackage[all]{xy}

\usepackage{tikz}

\usepackage{graphicx}
\usepackage{float}

\usepackage{xspace}

\usepackage[a4paper,body={16.3cm,22.8cm},centering]{geometry}
\usepackage[colorlinks,final,backref=page,hyperindex,pdftex]{hyperref}

\newcommand{\nc}{\newcommand}
\newcommand{\delete}[1]{}

\nc{\mlabel}[1]{\label{#1}}  
\nc{\mcite}[1]{\cite{#1}}  
\nc{\mref}[1]{\ref{#1}}  
\nc{\mbibitem}[1]{\bibitem{#1}} 

\delete{
\nc{\mlabel}[1]{\label{#1}  
{\hfill \hspace{1cm}{\small\tt{{\ }\hfill(#1)}}}}
\nc{\mcite}[1]{\cite{#1}{\small{\tt{{\ }(#1)}}}}  
\nc{\mref}[1]{\ref{#1}{{\tt{{\ }(#1)}}}}  
\nc{\mbibitem}[1]{\bibitem[\bf #1]{#1}} 
}
\newtheorem{theorem}{Theorem}[section]
\newtheorem{prop}[theorem]{Proposition}
\newtheorem{lemma}[theorem]{Lemma}
\newtheorem{coro}[theorem]{Corollary}

\theoremstyle{definition}
\newtheorem{defn}[theorem]{Definition}
\newtheorem{remark}[theorem]{Remark}
\newtheorem{exam}[theorem]{Example}
\newtheorem{prop-def}{Proposition-Definition}[section]

\newcommand\cal[1]{\mathcal{#1}}

\newcommand\alphlist{a,b,c,d,e,f,g,h,i,j,k,l,m,n,o,p,q,r,s,t,u,v,w,x,y,z}
\newcommand\Alphlist{A,B,C,D,E,F,G,H,I,J,K,L,M,N,O,P,Q,R,S,T,U,V,W,X,Y,Z}
\newcommand\getcmds[3]{\expandafter\newcommand\csname #2#1\endcsname{#3{#1}}}
\makeatletter
\@for\x:=\alphlist\do{\expandafter\getcmds\expandafter{\x}{frak}{\mathfrak}}
\@for\x:=\Alphlist\do{\expandafter\getcmds\expandafter{\x}{frak}{\mathfrak}}
\makeatother

\nc{\bfk}{{\bf k}}
\font\cyr=wncyr10

\newfont{\scyr}{wncyr10 scaled 550}
\nc{\sha}{\mbox{\cyr X}}
\nc{\ssha}{\mbox{\bf \scyr X}}

\nc{\id}{\mathrm{id}}
\nc{\Id}{\mathrm{Id}}
\nc{\lbar}[1]{\overline{#1}}
\nc{\ot}{\otimes}
\nc{\dep}{\mathrm{dep}}

\nc{\tred}[1]{\textcolor{red}{#1}} \nc{\tgreen}[1]{\textcolor{green}{#1}}
\nc{\tblue}[1]{\textcolor{blue}{#1}} \nc{\tpurple}[1]{\textcolor{purple}{#1}}

\nc{\li}[1]{\tpurple{\underline{Li:}#1 }}
\nc{\liadd}[1]{\tpurple{#1}}
\nc{\xing}[1]{\tblue{\underline{Xing:}#1 }}
\nc{\dominique}[1]{\tblue{\underline{Dominique: }#1 }}
\nc{\yuan}[1]{\tred{\underline{Yuan:}#1 }}
\nc{\markus}[1]{\tred{\underline{Markus:} #1}}
\nc\hu[1]{\tgreen{\underline{Huhu:}#1}}

\usetikzlibrary{calc}
\newlength\xch
\newsavebox\dbox
\sbox\dbox{\tikz{\fill (0,0) circle (0.05cm);}}
\newif\ifqdd
\newif\ifzdd
\nc{\dnx}{\Delta_n A} \nc{\dx}{\Delta A} \nc{\dgp}{{\rm deg_{P}}}
\nc{\dgt}{{\rm deg_{T}}} \nc{\dg}{{\rm deg}} \nc{\ida}{ID($A$)} \nc{\tu}{\tilde{u}} \nc{\tv}{\tilde{v}}
\nc{\nr}{\calr_n} \nc{\nz}{\calz_n} \nc{\fun}{\cala_{n,d}}
 \nc{\fbase}{\calb} \nc{\LF}{\mathrm{RF}} \nc{\FFA}{\mathrm{LF}} \nc{\irr}{\mathrm{Irr}}
 \nc{\result}{\bfk\mathrm{Irr}(S_n)}  \nc{\I}{I_{\mathrm{ID},n}^0}
 \nc{\nrs}{\calr_n^\star} \nc{\ii}{\mathrm{I}} \nc{\iii}{\mathrm{II}}
\nc{\intl}{{\rm int}}\nc{\ws}[1]{{#1}}\nc{\deleted}[1]{\delete{#1}}\nc{\plas}{placements\xspace}

\nc{\bim}[1]{#1}  \nc{\shaop}{\sha_{\Omega}^{+}}  \nc{\shao}{\sha_{\Omega}}
\nc{\bbim}[2]{#1 #2} \nc{\bbbim}[2]{#1,\, #2} \nc{\RBF}{{\rm RBF}}
\nc{\frb}{F_{\RB}} \nc{\shaf}{\ssha_{\tiny{\Omega}}} \nc{\sham}{\diamond_{\tiny{\Omega}}}
\nc{\lf}{\lfloor} \nc{\rf}{\rfloor} \nc{\shan}{\ssha_{\lambda}}
\nc{\rlex}{{\rm {lex}}} \nc{\bb}{\Box} \nc{\ra}{\rightarrow}
\nc{\e}{{\rm {e}}}
\nc{\DDF}{\mathrm{DD}(X,\,\Omega)}\nc{\DTF}{\mathrm{DT}(X,\,\Omega)} \nc{\DT}{\mathrm{DT}'(\Omega,\,V)}
\nc{\bra}{\mathrm{bra}} \nc{\bre}{\mathrm{bre}}
\nc{\dec}{\mathrm{dec}} \nc{\diamondw}{\diamond_{w}}
\nc{\type}{\mathrm{type}}

\nc\caF[1]{\cal{F}_{#1}(X,\,\Omega)}
\nc\calt{\cal{T}(X,\,\Omega)} \nc\caltn{\cal{T}_n(X,\,\Omega)}
\nc\caltbin{\cal{T}_b(X,\,\Omega)}
\nc\calta{\cal{T}_0(X,\,\Omega)}
\nc\caltb{\cal{T}_1(X,\,\Omega)}
\nc\caltc{\cal{T}_2(X,\,\Omega)}
\nc\caltd{\cal{T}_3(X,\,\Omega)}
\nc\caltm{\cal{T}_m(X,\,\Omega)}
\nc\calf{\cal{F}(X,\,\Omega)}
\nc\fram{\frak{M}(\Omega,\, X)}
\nc\shaw{\sha^{NC}_w(\Omega,\, X)}
\nc\dw{\diamond_w} \nc\dl{\diamond_\ell}
\nc\shal{\sha^{NC}_\ell(X,\, \Omega)} \nc\shav{\sha^{NC}_w(\Omega,\, V)} \nc\shat{\sha^{NC,1}_w(\Omega,\, T^{+}(V))}
\nc{\cfo}{\cal{F}(X,\,\Omega)}

\nc{\lar}{\varinjlim}
\nc\XO{(X,\,\Omega)}
\def\cxo#1#2;{\cal{#1}#2\XO}
\def\cxob#1#2;{\cal{#1}#2_b\XO}
\nc\lrf[2]{B_{#2}^+(#1)}
\nc{\fd}{\mathrm{\text{typed angularly decorated planar rooted trees}}}
\nc{\rb}{\mathrm{RBFWs}} \nc{\dfw}{\mathrm{DFW{(X)}}} \nc{\tfw}{\mathrm{TFW{(X)}}}
\nc{\tfv}{\mathrm{TFW{(V)}}} \nc{\rbf}{\mathrm{RBF}}
\nc{\db}{\mathrm{db}}
\nc{\st}{\mathrm{st}}

\def\Ve#1,#2,#3;{\vee_{#1,\,(#2,\,#3)}}
\def\bigv#1;#2;#3;{\bigvee\nolimits_{#1}^{#2;\,#3}}

\nc{\Irr}{\mathrm{Irr}} \nc{\lc}{\lfloor} \nc{\rc}{\rfloor}
\nc{\rswx}{\frak{M}( \Omega_R\sqcup \Omega_S, X)}
\nc{\rswxs}{\frak{M}^\star( \Omega_R\sqcup \Omega_S, X)}
\nc{\Dl}{\leq_{_{{\rm Dl}}}} \nc{\Dll}{<_{_{{\rm Dl}}}} \nc{\bbs}{\mathbb{S}}
\nc{\orbsa}{$\Omega$-Rota-Baxter system\xspace}
\nc{\orbsas}{$\Omega$-Rota-Baxter systems\xspace}
\nc{\mrbs}{matching Rota-Baxter system\xspace}
\nc{\mrbss}{matching Rota-Baxter systems\xspace}

\nc\prbsla[4]{{R}_{#1}\left(#3\right)R_{#2}\left(#4\right)}
\nc\prbsra[4]{{R}_{#1\rightarrow#2}\left(R_{#1\rhd#2}\left(#3\right)#4\right)+{R}_{#1\leftarrow#2}\left(#3S_{#1\lhd#2}\left(#4\right)\right)}
\nc\prbslb[4]{{S}_{#1}\left(#3\right)S_{#2}\left(#4\right)}
\nc\prbsrb[4]{{S}_{#1\rightarrow#2}\left(R_{#1\rhd#2}\left(#3\right)#4\right)+{S}_{#1\leftarrow#2}\left(#3S_{#1\lhd#2}\left(#4\right)\right)}

\nc\rbsla[4]{\lc #3 \rc ^{R}_{#1} \lc #4 \rc ^R_{#2}}
\nc\rbslb[4]{\lc #3\rc ^{S}_{#1}  \lc #4 \rc ^S_{#2}}
\nc\rbsray[4]{\lc \lc #3 \rc^R_{#1\rhd#2} #4 \rc ^{R}_{#1\rightarrow#2}}
\nc\rbsraz[4]{\lc #3 \lc #4\rc^S_{#1\lhd#2}\rc ^{R}_{#1\leftarrow#2}}
\nc\rbsrby[4]{\lc \lc #3\rc ^R_{#1\rhd#2}#4\rc ^{S}_{#1\rightarrow#2}}
\nc\rbsrbz[4]{\lc #3 \lc #4 \rc ^S_{#1\lhd#2}\rc ^{S}_{#1\leftarrow#2}}
\nc\rbsrac[4]{\lc #3 \lc #4\rc^R_{#1\lhd#2}\rc ^{R}_{#1\leftarrow#2}}

\nc\rbslq[4]{\lc #3 \rc ^{Q}_{#1} \lc #4 \rc ^Q_{#2}}
\nc\rbslt[4]{\lc #3\rc ^{T}_{#1}  \lc #4 \rc ^T_{#2}}
\nc\rbsrqy[4]{\lc \lc #3 \rc^R_{#1\rhd#2} #4 \rc ^{Q}_{#1\rightarrow#2}}
\nc\rbsrqz[4]{\lc #3 \lc #4\rc^S_{#1\lhd#2}\rc ^{Q}_{#1\leftarrow#2}}
\nc\rbsrty[4]{\lc \lc #3\rc ^R_{#1\rhd#2}#4\rc ^{T}_{#1\rightarrow#2}}
\nc\rbsrtz[4]{\lc #3 \lc #4 \rc ^S_{#1\lhd#2}\rc ^{T}_{#1\leftarrow#2}}

\nc{\obr}[1]{\lc #1 \rc_\omega^R} \nc{\obs}[1]{\lc #1 \rc_\omega^S} \nc{\obq}[1]{\lc #1 \rc_\omega^*}
\nc{\obqa}[1]{\lc #1 \rc_\alpha^*} \nc{\obqb}[1]{\lc #1 \rc_\beta^*}
\nc{\obra}[1]{\lc #1 \rc_\alpha^R} \nc{\obrb}[1]{\lc #1 \rc_\beta^R}
\nc{\obsa}[1]{\lc #1 \rc_\alpha^S} \nc{\obsb}[1]{\lc #1 \rc_\beta^S}

\begin{document}

\title[Free $\Omega$-Rota-Baxter systems and Gr\"obner-Shirshov bases]{Free $\Omega$-Rota-Baxter systems and Gr\"obner-Shirshov bases}

\author{Yuanyuan Zhang}
\address{School of Mathematics and Statistics, Henan University, Henan, Kaifeng 475004, P.\,R. China}
\email{zhangyy17@henu.edu.cn}
\author{Huhu Zhang} \address{School of Mathematics and Statistics,
Lanzhou University, Lanzhou, 730000, P. R. China}
\email{zhanghh20@lzu.edu.cn}
\author{Xing Gao}
\address{School of Mathematics and Statistics,
Key Laboratory of Applied Mathematics and Complex Systems,
Lanzhou University, Lanzhou, 730000, P.R. China}
\email{gaoxing@lzu.edu.cn}

\date{\today}

\begin{abstract}
In this paper, we propose the concept of an $\Omega$-Rota-Baxter system, which is a generalization of a Rota-Baxter system and an $\Omega$-Rota-Baxter algebra of weight zero. In the framework of operated algebras, we obtain a linear basis of a free $\Omega$-Rota-Baxter system for an extended diassociative semigroup $\Omega$, in terms of bracketed words and the method of Gr\"obner-Shirshov bases.
As applications, we introduce the concepts of Rota-Baxter system family algebras and matching Rota-Baxter systems as special cases of $\Omega$-Rota-Baxter systems, and construct their free objects. Meanwhile, free $\Omega$-Rota-Baxter algebras of weight zero, free Rota-Baxter systems, free Rota-Baxter family algebras and free matching Rota-Baxter algebras are reconstructed via new method.
\end{abstract}

\subjclass[2010]{
16W99, 
16S10, 
13P10, 
08B20, 
}

\keywords{Rota-Baxter family algebras, Rota-Baxter systems, $\Omega$-Rota-Baxter systems, Gr\"obner-Shirshov bases.}

\maketitle

\tableofcontents

\setcounter{section}{0}

\allowdisplaybreaks

\section{Introduction}
The concept of algebras with multiple linear operators (called $\Omega$-algebra) was first introduced by A. G. Kurosch in ~\cite{Kur} and there the author noticed that the free $\Omega$-algebra carries lots of combinatorial properties. Here $\Omega$ is a set to index the family of linear operators. As a key example of an $\Omega$-algebra, Rota-Baxter algebra (first called Baxter algebra) was introduced by Baxter~\cite{Bax} in his study on probability. Later some combinatoric properties of Rota-Baxter algebras were studied by Rota~\cite{Rot69} and Cartier~\cite{Car}.
Let $\bfk$ be a commutative ring and $\lambda\in \bfk$. A Rota-Baxter algebra of weight $\lambda$ is an associative $\bfk$-algebra with a Rota-Baxter operator $R:A\ra A$ satisfying
\[R(a)R(b)=R(aR(b)+R(a)b+\lambda a b),\,\text{ for }\, a,b\in A.\]
In particular, some scholars paid attention to the constructions of free Rota-Baxter algebras~\cite{Car, EFG08,  GK2, GK3, Guo12,Rot69}.
In recent years, $\Omega$-algebras have been studied extensively~\cite{Agu20, BoCh, Foi18, Foi20, FP, GGZ21, GZ, Guo09,ZGG,  ZGM, ZM}.\\

A Rota-Baxter system is a special $\Omega$-algebra with two linear operators, introduced by T. Brzezi\'{n}ski~\cite{Bre18} as an extension of Rota-Baxter algebra of weight $0$. There T. Brzezi\'{n}ski showed that dendriform algebra structures of a particular kind are equivalent to Rota-Baxter systems, and then he obtained that a Rota-Baxter system induces a weak pseudotwistor introduced by F. Panaite and F. V. Oystaeyen~\cite{PF}, which can be held responsible for the existence of a new associative product on the underlying algebra. J. J. Qiu and Y. Q. Chen~\cite{QC} obtained a linear basis of a free Rota-Baxter system on a set by using the method of Gr\"obner-Shirshov bases.\\

The Gr\"obner-Shirshov bases theory for Lie algebras was introduced by A. I. Shirshov~\cite{Shirshov}, in which the author defined the composition of two Lie polynomials and proved the Composition-Diamond lemma for the Lie algebras. Later, L. A. Bokut~\cite{b76} generalized the approach of Shirshov to associative algebras, see also G. M. Bergman~\cite{Bergman}. For commutative polynomials, this lemma is known as the Buchberger's Theorem~\cite{Buch65, Buch70}. Kurosh~\cite{Kur1} showed that any subalgebra of a free non-associative algebra is again free (the Nielsen-Schreier property).  Drensky and Holtkamp~\cite{DH} proved an analogue of the above Shirshov-Zhukov's Composition-Diamond lemma to free algebras for the case where all operations having aryties $2$.\\

Recently, L. Foissy and X. S. Peng~\cite{FP} introduced the concept of $\Omega$-Rota-Baxter algebras, which include Rota-Baxter family algebras and matching Rota-Baxter algebras as examples. In order to obtain free $\Omega$-Rota-Baxter algebras, they introduced  the notion of $\lambda$-extended diassociative semigroups, containing sets (for matching Rota-Baxter algebras) and semigroups (for Rota-Baxter family algebras) as special cases.\\

Along this line, in the present paper, we propose the concept of $\Omega$-Rota-Baxter systems, which include Rota-Baxter systems and $\Omega$-Rota-Baxter algebras of weight zero as special cases. The free object is obtained in terms of bracketed words and the method of Gr\"obner-Shirshov bases. As applications, the notations of Rota-Baxter system family algebras and matching Rota-Baxter systems are also introduced, and their free objects are obtained via Gr\"obner-Shirshov bases.
In particular, free $\Omega$-Rota-Baxter algebras of weight zero, free Rota-Baxter systems, free Rota-Baxter famlily algebras and free matching Rota-Baxter algebras are reconstructed. \\

The paper is organized as follows. In Section~\ref{sec:GS-bases}, we first propose the concept of an $\Omega$-Rota-Baxter system and give some examples. Then we construct an explicit monomial order (Proposition~\mref{pp:mord}) and obtain a linear basis of the free $\Omega$-Rota-Baxter system on a set $X$ (Theorem~\mref{thm:gsb}) for an extended diassociative semigroup $\Omega$, by applying the method of Gr\"obner-Shirshov bases.
As a direct consequence, we obtain a linear basis of the free $\Omega$-Rota-Baxter algebra of weight $0$ (Corollary~\mref{coro:forba}). Section~\ref{sec:application} is devoted to applications of the main result in Section~\mref{sec:GS-bases}.
We first propose the concept of a Rota-Baxter system family algebra as a special case of the $\Omega$-Rota-Baxter system, and construct its free object (Proposition~\mref{prop:RBF}). New methods to reconstruct free Rota-Baxter systems (Proposition~\mref{prop:frbs}) and free Rota-Baxter family algebras (Proposition~\mref{prop:gsbases}) are also given.
Second, we introduce the notion of a matching Rota-Baxter system as an example of an $\Omega$-Rota-Baxter system and build its free object (Proposition~\mref{prop:fmrb}). Finally, we reconstruct the free matching Rota-Baxter algebra (Proposition~\mref{prop:fmrba}).

\smallskip
{\bf Notation.}
Throughout this paper, we fix a commutative unitary ring $\bfk$,
which will be the base ring of all modules, algebras, tensor products, as well as linear maps. By an algebra we mean a unitary associative algebra, unless the contrary is specified.

\section{Free $\Omega$-Rota-Baxter systems}
\mlabel{sec:GS-bases}
In this section, we propose the concept of an $\Omega$-Rota-Baxter system and
construct its free object via the method of Gr\"obner-Shirshov bases.

\subsection{$\Omega$-Rota-Baxter systems}
In this subsection, we mainly define the notation of an $\Omega$-Rota-Baxter system and give some examples, which is simultaneously a generalization of a Rota-Baxter system~\cite{Bre18} and an $\Omega$-Rota-Baxter algebra of weight zero~\cite{FP}.
Let us first review these two concepts.

The concept of a Rota-Baxter system can help to understand the Jackson $q$-integral as a Rota-Baxter operator. It also extends~\cite{Bre18} the connections between three algebraic systems: Rota-Baxter algebras~\cite{Rot69}, dendriform algebras~\cite{Lod01} and infinitesimal bialgebras~\cite{Agu99}.

\begin{defn}\cite{Bre18}
 A triple $(A, R, S)$ consisting of an algebra $A$ and two linear operators $R, S: A \rightarrow A$ is called a {\bf Rota-Baxter system} if, for all $a,b \in A$,
\begin{align*}
R(a) R(b)=&\ R(R(a) b+a S(b)), \\
S(a) S(b)=&\ S(R(a) b+a S(b)).
\end{align*}
\mlabel{defn:mrba}
\end{defn}
\vskip-0.3in

The following is the notation of $\Omega$-Rota-Baxter algebras.

\begin{defn}\cite{FP}
Let $\Omega$ be a nonempty set equipped with five binary operations
$$\leftarrow,\rightarrow,\lhd,\rhd, \cdot: \Omega\times \Omega \rightarrow \Omega.$$
Let
$\lambda_{\Omega^2} := (\lambda_{\alpha,\,\beta})_{\alpha,\,\beta\in\Omega}$ be a collection of elements in $\bfk$.
A pair $(A, (R_{\omega})_{\omega\in\Omega})$ consisting of an algebra $A$ and a collection of linear operators $R_{\omega}: A \rightarrow A$, $\omega\in \Omega$ is called an {\bf $\Omega$-Rota-Baxter algebra of weight $\lambda_{\Omega^2}$ } if, for all $a,b \in A$,
$$R_\alpha(a)R_\beta(b)=R_{\alpha\rightarrow \beta}(R_{\alpha\rhd\beta}(a)b)+R_{\alpha\leftarrow\beta}(aR_{\alpha\lhd\beta}(b))
+\lambda_{\alpha,\,\beta} R_{\alpha\cdot\beta}(ab).$$
\end{defn}

Combining the above two notations, we propose the following concept studied in this paper.

\begin{defn}
Let $\Omega$ be a nonempty set equipped with four binary operations
$$\leftarrow,\rightarrow,\lhd,\rhd: \Omega\times \Omega \rightarrow \Omega.$$
A pair $(A, (R_\omega, S_\omega)_{\omega\in\Omega})$ consisting of an algebra $A$ and two collections of linear operators $R_\omega, S_\omega: A \rightarrow A$, $\omega\in \Omega$ is called an {\bf \orbsa} if, for all $a, b \in A$,
\begin{align}
\mlabel{eq:rbsy1} \prbsla{\alpha}{\beta}{a}{b}&=\prbsra{\alpha}{\beta}{a}{b}, \\
\mlabel{eq:rbsy2}\prbslb{\alpha}{\beta}{a}{b}&=\prbsrb{\alpha}{\beta}{a}{b}.
\end{align}
\mlabel{defn:omegarbsystem}
\end{defn}
\begin{remark}
In an \orbsa $(A, (R_\omega, S_\omega)_{\omega\in\Omega})$, if $R_{\omega}=S_\omega$ for $\omega\in \Omega$,
then we recover the definition of an $\Omega$-Rota-Baxter algebra of weight zero.
\mlabel{rk:orbs2orb}
\end{remark}

Extended diassociative semigroups can be used to study Gr\"{o}bner-Shirshov bases for \orbsas.

\begin{defn}\label{defn:diassociative semigroups}\mcite{Foi20,FP}
Let $\Omega$ be a nonempty set equipped with four binary operations $\leftarrow,\rightarrow,\lhd,\rhd:\Omega\times\Omega\to \Omega$. We say that $\Omega$ is an {\bf extended diassociative semigroup} if, for all $\alpha,\beta,\gamma\in\Omega$,
\allowdisplaybreaks
\begin{eqnarray*}
(\alpha\rightarrow\beta)\rightarrow\gamma&=&\alpha\rightarrow(\beta\rightarrow\gamma),\\
(\alpha\rightarrow\beta)\rhd\gamma&=&(\alpha\rhd(\beta\rightarrow\gamma))\rightarrow(\beta\rhd\gamma),\\
\alpha\rhd\beta&=&(\alpha\rhd(\beta\rightarrow\gamma))\rhd(\beta\rhd\gamma),\\
(\alpha\rightarrow\beta)\leftarrow\gamma&=&\alpha\rightarrow(\beta\leftarrow\gamma),\\
\alpha\rhd(\beta\leftarrow\gamma)&=&\alpha\rhd\beta,\\
(\alpha\rightarrow\beta)\lhd\gamma&=&\beta\lhd\gamma,\\
(\alpha\leftarrow\beta)\rightarrow\gamma&=&\alpha\rightarrow(\beta\rightarrow\gamma),\\
(\alpha\rhd(\beta\rightarrow\gamma))\leftarrow(\beta\rhd\gamma)&=&(\alpha\leftarrow\beta)\rhd\gamma,\\
(\alpha\rhd(\beta\rightarrow\gamma))\lhd(\beta\rhd\gamma)&=&\alpha\lhd\beta,\\
(\alpha\leftarrow\beta)\leftarrow\gamma&=&\alpha\leftarrow(\beta\rightarrow\gamma),\\
(\alpha\lhd\beta)\rightarrow((\alpha\leftarrow\beta)\lhd\gamma)&=&\alpha\lhd(\beta\rightarrow\gamma),\\
(\alpha\lhd\beta)\rhd((\alpha\leftarrow\beta)\lhd\gamma)&=&\beta\rhd\gamma,\\
(\alpha\leftarrow\beta)\leftarrow\gamma&=&\alpha\leftarrow(\beta\leftarrow\gamma),\\
(\alpha\lhd\beta)\leftarrow((\alpha\leftarrow\beta)\lhd\gamma)&=&\alpha\lhd(\beta\leftarrow\gamma),\\
(\alpha\lhd\beta)\lhd((\alpha\leftarrow\beta)\lhd\gamma)&=&\beta\lhd\gamma.\\
\end{eqnarray*}
\end{defn}

Enough examples show the vitality of a new concept. \orbsas include many familiar algebras as special cases.

\begin{exam}\label{exam:omega-system}
Let $(A, (R_\omega,S_\omega)_{\omega\in\Omega})$ be an $\Omega$-Rota-Baxter system.
\begin{enumerate}
\item \label{item:RB2} Let $\lambda\in \bfk$. If $S_\omega = R_{\omega}+\lambda\id$ for $\omega\in \Omega$, then Eqs.~(\ref{eq:rbsy1}-\ref{eq:rbsy2}) turn into

\begin{equation}
\begin{aligned}
R_\alpha(a)R_\beta(b)= R_{\alpha\to\beta}(R_{\alpha\lhd\beta}(a)b)
    +R_{\alpha\leftarrow\beta}(aR_{\alpha\lhd\beta}(b))
    +\lambda R_{\alpha\leftarrow\beta}(ab)
\end{aligned}
\label{eq:rbf1}
\end{equation}
and
\begin{equation}
\begin{aligned}
&\ R_\alpha(a)R_\beta(b)+\lambda R_\alpha(a)b+\lambda aR_\beta(b)+\lambda^2ab\\
=&\ R_{\alpha\rightarrow \beta}(R_{\alpha\rhd\beta}(a)b)
    +\lambda R_{\alpha\rhd\beta}(a)b
    +R_{\alpha\leftarrow\beta}(aR_{\alpha\lhd\beta}(b))\\
   &\ +\lambda R_{\alpha\leftarrow\beta}(ab)
    +\lambda aR_{\alpha\lhd\beta}(b)+\lambda^2ab.
\end{aligned}
\label{eq:rbf2}
\end{equation}
Further, if
$$\alpha\leftarrow\beta=\alpha\rightarrow\beta=:\alpha\cdot\beta\,\text{ and }\, \alpha\lhd\beta=\beta, \alpha\rhd\beta=\alpha,$$
then $(\Omega,\leftarrow,\rightarrow,\lhd,\rhd)$ is an extended diassociative semigroup and Eqs.~(\ref{eq:rbf1}-\ref{eq:rbf2}) degenerate into
\[R_\alpha(a)R_\beta(b)=R_{\alpha\cdot\beta}(R_\alpha(a)b+aR_\beta(b)+\lambda ab).\]
Thus,
$(A, (R_\omega)_{\omega\in \Omega})$ is a Rota-Baxter family algebra~\mcite{ZGM} of weight $\lambda$ with respect to the semigroup $(\Omega, \cdot)$.

\item \label{item:RB3} Let $\lambda_\omega\in \bfk$ for $\omega\in \Omega$. If $S_\omega = R_{\omega} + \lambda_\omega\id$ for each $\omega\in \Omega$,
   then Eqs.~(\ref{eq:rbsy1}-\ref{eq:rbsy2}) become
\begin{equation}
\begin{aligned}
R_\alpha(a)R_\beta(b)= R_{\alpha\to\beta}(R_{\alpha\lhd\beta}(a)b)
    +R_{\alpha\leftarrow\beta}(aR_{\alpha\lhd\beta}(b))
    +\lambda_{\alpha\lhd\beta} R_{\alpha\leftarrow\beta}(ab)
\end{aligned}
\label{eq:rbm1}
\end{equation}
and
\begin{equation}
\begin{aligned}
 S_\alpha(a)S_\beta(b)=&\ ((R_\alpha+\lambda_\alpha \id)(a)) \, ((R_\beta+\lambda_\beta\id)(b)) \\
    =&\ R_\alpha(a)R_\beta(b)+\lambda_\beta R_\alpha(a)b+\lambda_\alpha aR_\beta(b)+\lambda_\alpha\lambda_\beta ab \\
    =&\ R_{\alpha\rightarrow \beta}(R_{\alpha\rhd\beta}(a)b)
    +\lambda_{\alpha\rightarrow\beta} R_{\alpha\rhd\beta}(a)b
    +R_{\alpha\leftarrow\beta}(aR_{\alpha\lhd\beta}(b)) \\
    &\ +\lambda_{\alpha\lhd\beta} R_{\alpha\leftarrow\beta}(ab)
    +\lambda_{\alpha\leftarrow\beta} aR_{\alpha\lhd\beta}(b)+\lambda_{\alpha\leftarrow\beta}\lambda_{\alpha\lhd\beta}
    ab.
\end{aligned}
\label{eq:rbm3}
\end{equation}
Further, if
$$\alpha\rightarrow\beta=\alpha\lhd\beta=\beta\,\text{ and }\, \alpha\leftarrow\beta=\alpha\rhd\beta=\alpha,$$
then $(\Omega,\leftarrow,\rightarrow,\lhd,\rhd)$ is an extended diassociative semigroup and Eqs.~(\ref{eq:rbm1}-\ref{eq:rbm3}) reduce to
\[R_\alpha(a)R_\beta(b)=R_\alpha(aR_\beta(b))+R_\beta(R_\alpha(a)b)+\lambda_\beta R_\alpha(ab).\]
Thus,
$(A, (R_\omega)_{\omega\in \Omega})$ is a matching Rota-Baxter algebra~\mcite{FBP,ZGG} of weight $(\lambda_\omega)_{\omega\in \Omega}$.

\item Define $$a\prec_\alpha b:=aS_\alpha(b)\,\text{ and }\, a\succ_\alpha b:=R_\alpha(a)b.$$
Then we obtain an $\Omega$-dendriform algebra $(A,(\prec_\omega,\succ_\omega)_{\omega\in\Omega})$~\mcite{Foi20, FP}.
\end{enumerate}
\end{exam}

\subsection{Composition-Diamond lemma for free $\Omega$-operated algebras}
\label{sub:operated}
We are going to construct the free \orbsa, in the framework of operated algebras and via the method of Gr\"{o}bner-Shirshov bases.
Let us first recall the Composition-Diamond lemma for free $\Omega$-operated algebras~\cite{BoCh, GG}.

The concept of algebras with (one or more) linear operators was introduced by Kurosh~\cite{Kur}.
Later Guo~\mcite{Guo09} called such algebras operated algebras and constructed the free objects. See also~\cite{BC}.

\begin{defn}\mcite{Guo09}
Let $\Omega$ be a nonempty set.
\begin{enumerate}
\item An {\bf $\Omega$-operated algebra} is an algebra $A$ together with a set of linear operators $R_{\omega}:
A \rightarrow A,\,\omega \in \Omega.$

\item A morphism from an $\Omega$-operated algebra $(A, (R_{\omega})_{\omega \in \Omega})$ to an $\Omega$-operated algebra $(A', (R'_{\omega})_{\omega \in \Omega})$ is {\bf an algebra homomorphism} $f: A \rightarrow A'$ such that $f\circ R_{\omega}= R'_{\omega}\circ f$ for $\omega \in \Omega.$
\end{enumerate}
\end{defn}

The following is the construction of the free $\Omega$-operated algebra on a set $X$.
Denote by $M(X)$ the free monoid generated by $X$.
For any set $Y$ and $\omega\in \Omega$, let $\lfloor Y \rfloor_{\omega}$
denote the set $\big\{\lfloor y\rfloor_{\omega}\mid y\in Y\big\}$.
So $\lfloor Y \rfloor_{\omega}$ is a disjoint copy of $Y$.
Assume that the sets $\lfloor Y \rfloor_{\omega}$ to be disjoint with each other when $\omega$ varies in $\Omega.$
We now use induction to define a direct system $\big\{\frak{M_{n}}=\frak{M_{n}}(\Omega, X), i_{n,\, n+1}: \frak{M_{n}} \rightarrow \frak{M}_{n+1}\big\}_{n\geq 0}$ of free monoids. We first define
\begin{equation*}
\frak{M}_{0}:=M(X)\,\text{ and }\, \mathfrak{M}_{1}:= M\left(X\sqcup (\sqcup_{\omega\, \in \Omega}\lfloor \frak{M}_{0} \rfloor_{\omega})\right),
\end{equation*}
with $i_{0,\,1}$ being the inclusion
\begin{equation*}
i_{0,\,1}: \frak{M}_{0}=M(X)\hookrightarrow  \frak{M}_{1}=M\left(X\sqcup (\sqcup_{\omega\,\in \Omega}\lfloor \frak{M}_{0}\rfloor_{\omega})\right).
\end{equation*}
Inductively assume that $\frak{M}_{n-1}$ has been defined for given $n \geq 2,$ with the inclusion
\begin{equation}
i_{n-2,\,n-1}: \frak{M}_{n-2}\rightarrow \frak{M}_{n-1}.
\mlabel{eq:inclu2}
\end{equation}
We then define
\begin{equation*}
\frak{M}_{n}:=M\left(X\sqcup(\sqcup_{\omega\,\in \Omega}\lfloor \frak{M}_{n-1}\rfloor_{\omega})\right).
\end{equation*}
The inclusion in Eq.~(\mref{eq:inclu2}) induces the inclusion
$$\lfloor \frak{M}_{n-2}\rfloor_{\omega}\rightarrow \lfloor \frak{M}_{n-1}\rfloor_{\omega}, \,\text{ for each }\,\omega\in \Omega,$$
which generates an inclusion  of free monoids
$$i_{n-1,\,n}: \frak{M}_{n-1}=M\left(X\sqcup(\sqcup_{\omega \,\in \Omega}\,\lfloor \frak{M}_{n-2}\rfloor_{\omega})\right)
\hookrightarrow M\left(X\sqcup(\sqcup_{\omega \,\in \Omega}\,\lfloor \frak{M}_{n-1}\rfloor_{\omega})\right)=\frak{M}_{n}.$$
This completes the inductive
construction of the direct systems. Define the direct limit of monoids
$$ \frak{M}(\Omega,\, X) :=\lim_{\longrightarrow}\frak{M}_{n}=\bigcup_{n\geq 0}\frak{M}_{n}$$
with identity $1$.

Let us collect some basic concepts used later.

\begin{defn} Let $X$ be a set and $\Omega$ a nonempty set.
\begin{enumerate}
\item Elements of $\frak{M}_{n}\backslash \frak{M}_{n-1}$ are said to have {\bf depth} $n$.

\item Elements in $\frak{M}(\Omega,\, X)$ (resp. $\bfk\frak{M}(\Omega,\, X)$) are called {\bf bracketed words} (resp. {\bf bracketed polynomials}) on $X$.

\item  If $u\in X\sqcup (\sqcup_{\omega\in\Omega}\lc \frak{M}(\Omega,\, X)\rc_\omega)$, we call $u$ {\bf prime}. For $u=u_1\cdots
u_n\in\frak{M}(\Omega,\, X)$ with each $u_i$ prime, we define the {\bf breadth} $|u|$ of $u$ to be $|u|:=n$.
Here we employ the convention that $|1|:=0$.
\end{enumerate}
\end{defn}

Denote by $\bfk\frak{M}(\Omega,\, X)$ the free module on $\frak{M}(\Omega,\, X)$. Extending by linearity, the multiplication on
$\frak{M}(\Omega,\, X)$ can be extended to $\bfk\frak{M}(\Omega,\, X)$, turning it into an algebra.
For each $\omega\in \Omega$, the operator
$$\lfloor  \, \rfloor_{\omega}:\frak{M}(\Omega,\, X)\rightarrow \frak{M}(\Omega,\, X), \
w\mapsto \lc w\rc_\omega$$
can be extended linearly to a linear operator on $\bfk\frak{M}(\Omega,\, X)$, still denoted by $\lc \, \rc_\omega$.
The $\Omega$-operated algebra $(\bfk\frak{M}(\Omega,\, X),\, \{\lf\, \rf_{\omega} \mid \omega\in \Omega\} )$ is
indeed the free object in the category of $\Omega$-operated algebras.

\begin{lemma} \cite{Guo09}
Let $X$ be a set and $\Omega$ a nonempty set. Let $\iota: X\rightarrow \bfk\frak{M}(\Omega,\, X)$ be the natural embedding.
Then the pair $(\bfk\frak{M}(\Omega,\, X),\, \{\lf\, \rf_{\omega} \mid \omega\in \Omega\} )$, together with the embedding $\iota$, is
the free $\Omega$-operated algebra on $X$.
\end{lemma}

The $\star$-bracketed words are used in the theory of Gr\"{o}bner-Shirshov bases.

\begin{defn}
Let $X$ be a set and $\star$ a symbol not in $X$.
\begin{enumerate}
\item By a {\bf $\star$-bracketed word} on $X$, we mean any bracketed word in $\mathfrak{M}(\Omega,\,X\sqcup \{\star\})$ with exactly one occurrence of $\star$, counting multiplicities. The set of all $\star$-bracketed words on $X$ is denoted by $\mathfrak{M}^{\star}(\Omega,\,X)$.

    \item For $q\in \mathfrak{M}^{\star}(\Omega,\,X)$ and $u\in \mathfrak{M}(\Omega,\,X)$, we define $q|_u:= q|_{\star\mapsto u}$ to be the bracketed word on $X$ obtained by replacing the symbol $\star$ in $q$ by $u$.

\item For $q\in \mathfrak{M}^{\star}(\Omega,\,X)$ and $s=\Sigma_ic_iq|_{u_i}\in \bfk\mathfrak{M}(\Omega,\,X),$ where $c_i\in \bfk$ and $u_i\in \mathfrak{M}(\Omega,\,X)$, we define
    $$q|_s:=\Sigma_ic_iq|_{u_i}.$$
\end{enumerate}
\end{defn}

For example, if
$ q=\lc x  \lc y\star  z\rc_{\omega_1} \rc_{\omega_2}\in\mathfrak{M}^{\star}(\Omega,\,X),
$
then
$
q|_u =  \lc x  \lc yu  z\rc_{\omega_1} \rc_{\omega_2}.
$

\begin{defn}
Let $X$ be a set and $\Omega$ a nonempty set. {\bf A monomial order} on $\mathfrak{M}(\Omega,\,X)$ is a {\bf well order} $\leq$ on $\mathfrak{M}(\Omega,\,X)$ such that 
\begin{equation*}
u < v \Rightarrow q\mid_{u} < q\mid_{v},\, \text{for all}\,\, u,v \in \mathfrak{M}(\Omega,\,X)\,\text{and all}\, q\in\frak{M}^\star(\Omega,\, X).
 \mlabel{eq:order}
\end{equation*}
Here, as usual, we denote $u < v$ if $u\leq v$ but $u\neq v.$%
\end{defn}

\begin{defn}
Let $\leq$ be a monomial order on $\frak{M}(\Omega,\, X)$ and $f, g \in \bfk\frak{M}(\Omega,\, X)$ two distinct monic bracketed polynomials.
\begin{enumerate}
\item
If there exist $u, v, w \in
\mathfrak{M}(\Omega,\,X)$ such that $w=\lbar{f}u=v\lbar{g}$ with $\max \big\{|\lbar{f}|, |\lbar{g}|\big\} < w<|\lbar{f}|+|\lbar{g}|$, we call
$$(f, g)_{w}:=(f, g)_{w}^{u,v}:=fu-vg$$
the {\bf intersection composition of $f$ and $g$ with respect to} $(u,v)$.

\item If there exist $q \in \frak{M}^\star(\Omega,\, X)$ and $w \in \mathfrak{M}(\Omega,\,X)$ such  that
$w=\lbar{f}=q|_{\lbar{g}},$ we call
$$(f, g)_{w}:=(f, g)_{w}^q:=f-q|_{g}$$
the {\bf including composition of $f$ and $g$ with respect to} $q$.
\end{enumerate}
\mlabel{defn:composi}
\end{defn}

The $w$ in Definition~\mref{defn:composi} are called {\bf ambiguities with respect to $f$ and $g$}.
Now we are ready for the concept of Gr\"{o}bner-Shirshov bases.

\begin{defn}
Let $\leq$ be a monomial order on $\mathfrak{M}(\Omega,\,X)$,
 $\bbs\subseteq \bfk\frak{M}(\Omega,\, X)$
 a set of monic bracketed polynomials and $w \in \mathfrak{M}(\Omega,\,X).$
\begin{enumerate}
\item For $u, v\in \bfk\frak{M}(\Omega,\, X),$ we call $u$ and $v$ are {\bf congruent modulo} $(\bbs, w)$ and denote this by
    $$u\equiv v\, \text{ mod }\, (\bbs, w)$$
    if $u-v=\Sigma_{i}c_{i}q_{i}|_{s_{i}},$ where $c_{i}\in \bfk\setminus\{0\}, q_{i}\in \frak{M}^\star(\Omega,\, X), s_{i}\in \bbs$ and $q_{i}|_{\lbar{s_{i}}}<w.$

\item For $f, g \in \bfk\frak{M}(\Omega,\, X)$ and suitable $w$, $u$, $v$ or $q$ that gives an intersection composition $(f, g)_{w}^{u,v}$ or an including composition $(f, g)_{w}^{q}$, the composition is called {\bf trival modulo $(\bbs,w)$}
if
$$(f, g)_{w}^{u,v}\,\text{ or }\, (f, g)_{w}^{q}\equiv 0\,\text{ mod }\,(\bbs, w).$$

\item The set {$\bbs$ is called a {\bf Gr\"{o}bner-Shirshov bases} with respect to $\leq$} if, for all pairs $f, g \in \bbs$, all intersection compositions $(f, g)_{w}^{u,v}$ and all including compositions $(f, g)_{w}^q$ are trivial modulo $(\bbs, w).$
\end{enumerate}
\end{defn}

The following result is the well-known Composition-Diamond lemma for $\Omega$-operated algebras.

\begin{theorem}\cite{BoCh, GSZ}\label{Composition-Diamond lemma}
Let  $\leq$ be a monomial order on $\frak{M}(\Omega,\, X)$ and $\bbs$ a set of monic bracketed polynomials in $\bfk\frak{M}(\Omega,\, X)$.  Then the following statements are equivalent:
 \begin{enumerate}
\item[(I)] $\bbs $ is a Gr\"{o}bner-Shirshov basis in $\bfk\frak{M}(\Omega,\, X)$.

\item[(II)] If $ f\in \Id(\bbs)$, then $\lbar{f}=q|_{\overline{s}}$
for some $q \in \frak{M}^{\star}(\Omega,\, X)$ and $s\in \bbs$.

\item[(II')]  If $f\in \Id(\bbs)$, then $f=\alpha_1q_1|_{s_1}+\cdots+\alpha_nq_n|_{s_n}$,
for some $q_i \in \frak{M}^{\star}(\Omega,\, X)$ and some $s_i\in \bbs$ with
$q_1|_{\overline{s_1}}>\cdots>q_n|_{\overline{s_n}}$.

\item[(III)] $\bfk\mathfrak{M}(\Omega,\,X)=\bfk\operatorname{Irr}(\bbs) \oplus \operatorname{Id}(\bbs)$,  where
$$
\operatorname{Irr}(\bbs)=\mathfrak{M}(\Omega, X) \backslash\left\{\left.q\right|_{\lbar{s}} \mid q \in \mathfrak{M}^{\star}(\Omega, X), s \in \bbs\right\},
$$
and $\operatorname{Irr}(\bbs)$ is a $\bfk$-basis of $\bfk \mathfrak{M}(\Omega, X) / \operatorname{Id}(\bbs)$.
\end{enumerate}
\end{theorem}

\subsection{Gr\"{o}bner-Shirshov bases for $\Omega$-Rota-Baxter systems}
In this subsection, we first construct a required monomial order on $\frakM(\Omega, X)$. Then by the Composition-Diamond lemma for $\Omega$-operated algebras, we obtain a linear basis of the free \orbsa.

\smallskip

{\bf Notice.} Let $\Omega_R$ and $\Omega_S$ be two disjoint copies of $\Omega$.
{\em In the rest of this paper}, in order to distinguish the linear operators appearing in $\rswx$, we denote
\begin{align*}
\obr{\,}: \bfk \rswx \to \bfk\rswx, w\mapsto \lc w\rc_{\omega}, \,\text{ for }\,  \omega\in \Omega_R,\\
\obs{\,}: \bfk \rswx \to \bfk\rswx, w\mapsto \lc w\rc_{\omega},  \,\text{ for }\,\omega\in \Omega_S.
\end{align*}
Write $\obq{\,}$ to be $\obr{\,}$ or $\obs{\,}$.

Let $(X, \leq_X)$ and $(\big\{\obr{\,}, \obs{\,} \mid \omega \in \Omega\big\}, \leq_\Omega)$ be two well-ordered sets.
We now extend $\leq_X$ and $\leq_\Omega$ to a monomial order $\leq_\db$ on $\rswx$.
Let $u\in \rswx$. Define ${\rm deg}(u)$ to be the number of all occurrences of all $x\in X$ and $\obq{\,} \in \big\{\obr{\,}, \obs{\,} \mid \omega \in \Omega\big\}$, counting multiplicity. Writting $u=u_1\cdots u_n\in \rswx$ with $n\geq 1$ and each $u_i$ prime, denote by
\begin{equation*}
\st(u):=(u_1,\cdots, u_n)\,\text{ and }\, {\rm wt}(u) :=(\mbox{deg}(u), |u|, u_1,\cdots, u_n). \mlabel{eq:dborder}
\end{equation*}

For $u,v\in \rswx$, define $u\leq_\db v$ inductively on $\dep(u)+\dep(v)\geq 0$. For the initial step of $\dep(u)+\dep(v)=0$, we have $u, v\in M(X)$ and define $u\leq_{\db} v$ by the degree lexicographical order, that is,
$$
u \leq_{\db} v \, \text{ if }\, {\rm wt}(u)\leq {\rm wt}(v)\quad  \mbox{lexicographically}.
$$
Here notice that $\deg(u)=|u|$ and $\deg(v)=|v|$.
For the induction step, we first assume $|u|=|v|=1$.
If $u= \lc \tilde{u}\rc_\alpha^{\ast_1}$ and $v= \lc \tilde{v}\rc_\beta^{\ast_2}$ for some $ \lc \,\rc_\alpha^{\ast_1}, \lc \,\rc_\beta^{\ast_2}\in \big\{\obr{\,}, \obs{\,} \mid \omega \in \Omega\big\}$ and some
$\tilde{u}, \tilde{v}\in \rswx$, then define
\begin{equation}
u\leq_{\db} v \,\text{ if }\, (\lc \,\rc_\alpha^{\ast_1}, \tilde{u} )\leq (\lc \,\rc_\beta^{\ast_2}, \tilde{v})\quad  \mbox{lexicographically}.  \mlabel{eq:induction}
\end{equation}
Here we use $\leq_\Omega$ for the first component and induction hypothesis for the second component.
If $u\in X$ and $v= \obqb{\tilde{v}}$ for some $\beta\in \Omega$ and $\tilde{v}\in \rswx$, then define $u<_\db v$.
Next, for general $u,v\in \rswx$, we define
$u\leq_\db v$ by
\begin{equation}
u\leq_\db v\Leftrightarrow
\left\{\begin{array}{ll}
\deg(u)<\deg(v), \\
\text{ or } \deg(u) = \deg(v)\, \text{ and }\, |u|<|v|,\\
\text{ or } \deg(u) = \deg(v),\, |u|= |v| \, \text{ and }\st(u) \leq \st(v) \quad \text{ lexicographically.}
\end{array}
\right.
\mlabel{eq:db0}
\end{equation}
Namely, we define
\begin{equation}
u\leq_\db v \,\text{ if }\, {\rm wt}(u) \leq {\rm wt}(v)\quad  \text{lexicographically. }
\mlabel{eq:db}
\end{equation}

We expose the following useful facts.

\begin{lemma}
\begin{enumerate}
\item\cite{GSZ} Let $A$ and B be two well-ordered sets. Then we obtain an extended well order on the disjoint union $A \sqcup B$ by defining $a<b$ for all $a \in A$ and $b \in B$. \mlabel{GSZ}
\item\cite{Har} Let $\leq_{Y_{i}}$ be a well order on $Y_{i}, 1 \leq i \leq k, k \geq 1$. Then the lexicographical product order is a well order on the cartesian product $Y_{1} \times \cdots \times Y_{k}$.  \mlabel{Har}
\end{enumerate}
\mlabel{lemma:order}
\end{lemma}

Now we are ready to prove that the order $\leq_\db$ is a monomial order.

\begin{prop}
Let $(X, \leq_X)$ and $(\big\{\obr{\,}, \obs{\,} \mid \omega \in \Omega\big\}, \leq_\Omega)$ be two well-ordered sets.
The order $\leq_\db$ defined above is a monomial order on $\rswx$.
\mlabel{pp:mord}
\end{prop}

\begin{proof}
We first prove that $\leq_\db$ is a well order on $\rswx$.
The restriction of $\leq_\db$ on $M(X)$ is the degree lexicographical order, which is a well order~\mcite{BN}.
The restriction of $\leq_\db$ on
$$\big\{ \obq{\rswx} \mid \omega\in \Omega \big\} = \big\{ \lc \rswx \rc_\omega^R, \lc \rswx \rc_\omega^S \mid \omega\in \Omega \big\}$$ is a well order by Eq.~(\mref{eq:induction}), Lemma~\mref{lemma:order}-\mref{Har} and induction on the sum of depth. By Lemma~\mref{lemma:order}-\mref{GSZ}, the restriction of $\leq_\db$ on the set of prime elements
$X\sqcup \big\{ \obq{\rswx} \mid \omega\in \Omega \big\}$ is a well order.
Finally, since $\deg(u), |u|\in \mathbb{Z}_{\geq 0}$, the order $\leq_\db$ is a well order on $\rswx$ by Eq.~(\mref{eq:db}) and Lemma~\mref{lemma:order}-\mref{Har}.

We are left to verify that the $\leq_\db$ are compatible with the linear operators $\lc\,\rc^\ast_\omega\in\{\lc\,\rc^R_\omega, \lc\,\rc^S_\omega\mid\omega\in\Omega\}$ and the concatenation product. The former follows from Eq.~\eqref{eq:induction} by taking $\lc \,\rc_\alpha^{\ast_1}:= \lc \,\rc_\beta^{\ast_2}:= \lc \,\rc_\omega^{\ast}$.
For the later, it suffices to prove the implication
$$
u \leq_\db v \Longrightarrow w u \leq_\db w v\,\text{ and }\,  uw \leq_\db vw \text{ for }\, w \in \rswx .
$$
By symmetry, we only prove the case of $w u \leq_\db w v$ provided $u \leq_\db v$.
There are three cases to consider according to Eq.~(\mref{eq:db0}).

\noindent {\bf Case 1.} $\deg(u)<\deg(v)$. Then
\begin{align*}
\deg(w u)=\deg(u)+\deg(w)<\deg(v)+\deg(w)=\deg(w v),
\end{align*}
and so $w u<_\db w v$ by Eq.~\eqref{eq:db}.

\noindent{\bf Case 2.} $\deg(u)=\deg(v)$ and $|u|<|v|$. In this case, we have
$\deg(w u)=\deg(w v)$ and
 \[|w u|=|w|+|u|<|w|+|v|=|w v|,\]
which implies $w u<_\db w v$.

\noindent{\bf Case 3.} $\deg(u)=\deg(v),\,|u|=|v|$ and $\st(u) \leq \st(v)$ lexicographically. Then
$\deg(w u)=\deg(w v)\,\text{ and }\, |w u|=|w v|.$
Write
$$u=u_1\cdots u_m,\, v=v_1\cdots v_m\,\text{ and }\, w=w_1\cdots w_n,$$
where all $u_i$, $v_j$ and $w_k$ are prime. Since
$$\st(u)=(u_1,\cdots, u_m) \leq \st(v)=(v_1,\cdots, v_m)\quad\text{lexicographically}, $$
we have
$$\st(wu)=(w_1,\cdots ,w_n, u_1,\cdots, u_m) \leq \st(wv)=(w_1,\cdots ,w_n,v_1,\cdots, v_m)\quad\text{lexicographically }.$$
Thus we have $w u \leq_\db w v$. This completes the proof.
\end{proof}

Now we arrive at our first main result of this paper.

\begin{theorem}
Let $X$ be a set and $\Omega$ a set with four binary operations $\leftarrow,\rightarrow,\lhd$ and $\rhd$.
Let $\leq_\db$ be the monomial order on $\rswx$ defined as above.
\begin{enumerate}
\item The set
\begin{equation*}
\mathbb{S}_{\Omega\, S}:=\left\{
\left.
 \begin{array}{ll}
\rbsla{\alpha}{\beta}{u}{v}-\rbsray{\alpha}{\beta}{u}{v}-\rbsraz{\alpha}{\beta}{u}{v}\\
\rbslb{\alpha}{\beta}{u}{v}-\rbsrby{\alpha}{\beta}{u}{v}-\rbsrbz{\alpha}{\beta}{u}{v} \\
\end{array}
\right|u, v\in \rswx\,\text{ and }\, \alpha, \beta\in\Omega \right\}
\end{equation*}
is a Gr\"{o}bner-Shirshov basis in $\bfk\rswx$ if and only if
$(\Omega, \leftarrow,\rightarrow,\lhd,\rhd)$ is an extended diassociative semigroup. \mlabel{it:orbsaa}

\item If $(\Omega, \leftarrow,\rightarrow,\lhd,\rhd)$ is an extended diassociative semigroup, then the set
$$\Irr(\mathbb{S}_{\Omega\, S}) := \{ w\in \rswx \mid  w \neq
q|_{\overline{s}}\,  \text{ for  any }\, q \in \rswxs \,  \text{ and any }\, s\in \mathbb{S}_{\Omega\, S}\}$$
is a $\bfk$-basis of the free \orbsa $\bfk\rswx/\Id(\mathbb{\mathbb{S}}_{\Omega\, S})$. \mlabel{it:orbsab}
\end{enumerate}
\mlabel{thm:gsb}
\end{theorem}

\begin{proof}
\mref{it:orbsaa}
 For $u,v\in \rswx$ and $\alpha,\beta\in\Omega$, write
$$
f_{\alpha,\,\beta}(u,v):= \rbsla{\alpha}{\beta}{u}{v}-\rbsray{\alpha}{\beta}{u}{v}-\rbsraz{\alpha}{\beta}{u}{v},
$$
$$
g_{\alpha,\,\beta}(u, v):= \rbslb{\alpha}{\beta}{u}{v} -\rbsrby{\alpha}{\beta}{u}{v}-\rbsrbz{\alpha}{\beta}{u}{v}.
$$
With respect to $\leq_\db$, the leading monomials of $f_{\alpha,\,\beta}(u,v)$ and $g_{\alpha,\,\beta}(u, v)$ are
$\rbsla{\alpha}{\beta}{u}{v}$ and $\rbslb{\alpha}{\beta}{u}{v}$, respectively.
All possible compositions are listed as below:
\begin{eqnarray*}
\text{intersection compositions},&\ \ &\text{ambiguities}\\
(f_{\alpha,\,\beta}(u,v), f_{\beta,\,\gamma}(v, w))_{w_1},&  \ \  &w_1=\lc u \rc ^R_\alpha \lc v\rc ^R_\beta \lc w\rc ^R_\gamma,\nonumber\\
(g_{\alpha,\,\beta}(u,v), g_{\beta,\,\gamma}(v, w))_{w_2},&  \ \  &w_2=\lc u \rc ^S_\alpha \lc v\rc ^S_\beta \lc w\rc ^S_\gamma,\nonumber\\ \\
\text{including compositions},&\ \ &\text{ambiguities}\\
(f_{\gamma,\,\delta}(q|_{{\rbsla{\alpha}{\beta}{u}{v}}}, w), f_{\alpha,\,\beta}(u, v))_{w_3},&\ \  &w_3=\rbsla{\gamma}{\delta}{q|_{{\rbsla{\alpha}{\beta}{u}{v}}}}{w},\nonumber\\
(f_{\alpha,\,\delta}(u, q|_{{\rbsla{\beta}{\gamma}{v}{w}}}), f_{\beta,\,\gamma}(v, w))_{w_4},&\ \  &w_4=\rbsla{\alpha}{\delta}{u}{q|_{{\rbsla{\beta}{\gamma}{v}{w}}}},\nonumber\\
(f_{\gamma,\,\delta}(q|_{{{\rbslb{\alpha}{\beta}{u}{v}}}}, w), g_{\alpha,\,\beta}(u, v))_{w_5},&\ \  &w_5=\rbsla{\gamma}{\delta}{q|_{{\rbslb{\alpha}{\beta}{u}{v}}}}{w},\nonumber\\
(f_{\alpha,\,\delta}(u, q|_{{\rbslb{\beta}{\gamma}{v}{w}}}), g_{\beta,\,\gamma}(v, w))_{w_6},&\ \  &w_6=\rbsla{\alpha}{\delta}{u}{q|_{{\rbslb{\beta}{\gamma}{v}{w}}}},\nonumber\\
(g_{\gamma,\,\delta}(q|_{{{\rbsla{\alpha}{\beta}{u}{v}}}}, w), f_{\alpha,\,\beta}(u, v))_{w_7},&\ \  &w_7=\rbslb{\gamma}{\delta}{q|_{{\rbsla{\alpha}{\beta}{u}{v}}}}{w},\nonumber\\
(g_{\alpha,\,\delta}(u, q|_{{\rbsla{\beta}{\gamma}{v}{w}}}), f_{\beta,\,\gamma}(v, w))_{w_{8}},&\ \  &w_{8}=\rbslb{\alpha}{\delta}{u}{q|_{{\rbsla{\beta}{\gamma}{v}{w}}}},\nonumber\\
(g_{\gamma,\,\delta}(q|_{{\rbslb{\alpha}{\beta}{u}{v}}}, w), g_{\alpha,\,\beta}(u, v))_{w_9},&\ \  &w_9=\rbslb{\gamma}{\delta}{q|_{{\rbslb{\alpha}{\beta}{u}{v}}}}{w},\nonumber\\
(g_{\alpha,\,\delta}(u, q|_{{\rbslb{\beta}{\gamma}{v}{w}}}), g_{\beta,\,\gamma}(v, w))_{w_{10}},&\ \  &w_{10}=\rbslb{\alpha}{\delta}{u}{q|_{{\rbslb{\beta}{\gamma}{v}{w}}}}.\nonumber
\label{composition}
\end{eqnarray*}

Among these ambiguities, there are five pairs $(w_1, w_2)$, $(w_3, w_4)$, $(w_5, w_6)$, $(w_7, w_8)$ and $(w_9, w_{10})$.
The pair $(w_1, w_2)$ is symmetric by exchanging $\lc \, \rc ^R_\omega$ and $\lc \, \rc ^S_\omega$ for each $\omega\in \Omega$. The pairs $(w_3, w_4)$, $(w_5, w_6)$, $(w_7, w_8)$ and $(w_9, w_{10})$ are symmetric
in the sense that the ambiguity of one composition in a pair can be obtained
from the ambiguity of the other composition by taking the opposite multiplication. Hence for each
pair, it suffices to show the triviality of the composition from the first ambiguity.
Compositions from $w_1$ and $w_2$ are trivial if and only if $\Omega$ is an extended diassociative semigroup,
and others are trivial automatically.

Indeed, for the first one, we have
\allowdisplaybreaks{
\begin{eqnarray*}
&&\Big(f_{\alpha,\,\beta}(u,v), f_{\beta,\,\gamma}(v, w)\Big)_{w_1}\\
&=&f_{\alpha,\,\beta}(u,v)\lc w \rc^R_{\gamma}-\lc u \rc ^R_{\alpha}f_{\beta,\,\gamma}(v, w)\\
&=&
\left(\rbsla{\alpha}{\beta}{u}{v}-\rbsray{\alpha}{\beta}{u}{v}-\rbsraz{\alpha}{\beta}{u}{v}\right)\lc w \rc^R_{\gamma}\\
&& -\lc u \rc ^R_{\alpha}\left(\rbsla{\beta}{\gamma}{v}{w}-\rbsray{\beta}{\gamma}{v}{w}-\rbsraz{\beta}{\gamma}{v}{w}\right)\\
&=& -\rbsray{\alpha}{\beta}{u}{v}\lc w \rc^R_{\gamma} - \rbsraz{\alpha}{\beta}{u}{v}\lc w \rc^R_{\gamma}+  \lc u \rc ^R_{\alpha}\rbsray{\beta}{\gamma}{v}{w} +  \lc u \rc ^R_{\alpha}\rbsraz{\beta}{\gamma}{v}{w}\\
&\equiv&
- \rbsray{(\alpha\rightarrow\beta)}{\gamma}{\lc u \rc ^R_{\alpha\rhd\beta} v}{w}-\rbsraz{(\alpha\rightarrow\beta)}{\gamma}{\lc u \rc ^R_{\alpha\rhd\beta} v}{w}\\
&&-\rbsray{(\alpha\leftarrow\beta)}{\gamma}{u\lc v \rc ^S_{(\alpha\lhd\beta)}}{w}-\rbsraz{(\alpha\leftarrow\beta)}{\gamma}{u\lc v \rc ^S_{(\alpha\lhd\beta)}}{w}\\
&&+\rbsray{\alpha}{(\beta\rightarrow\gamma)}{u}{\lc v \rc^R_{\beta\rhd\gamma}w}+\rbsraz{\alpha}{(\beta\rightarrow\gamma)}{u}{\lc v \rc^R_{\beta\rhd\gamma}w}\\
&&+\rbsray{\alpha}{(\beta\leftarrow\gamma)}{u}{v \lc w \rc^S_{\beta\lhd\gamma}}+\rbsraz{\alpha}{(\beta\leftarrow\gamma)}{u}{v \lc w \rc^S_{\beta\lhd\gamma}}\\
&\equiv&
-\rbsray{(\alpha\rightarrow\beta)}{\gamma}{\lc u \rc ^R_{\alpha\rhd\beta} v}{w}-\rbsraz{(\alpha\rightarrow\beta)}{\gamma}{\lc u \rc ^R_{\alpha\rhd\beta} v}{w}-\rbsray{(\alpha\leftarrow\beta)}{\gamma}{u\lc v \rc ^S_{(\alpha\lhd\beta)}}{w}\\
&&- \lc u \rbsrby{(\alpha\lhd\beta)}{((\alpha\leftarrow\beta)\lhd\gamma)}{v}{w}\rc ^R_{(\alpha\leftarrow\beta)\leftarrow\gamma}
-\lc u\rbsrbz{(\alpha\lhd\beta)}{((\alpha\leftarrow\beta)\lhd\gamma)}{v}{w} \rc^R_{(\alpha\leftarrow\beta)\leftarrow\gamma}\\
&&+\lc \rbsray{(\alpha\rhd(\beta\rightarrow\gamma))}{(\beta\rhd\gamma)}{u}{v} w \rc ^R_{\alpha\rightarrow(\beta\rightarrow\gamma)}
+\lc \rbsraz{(\alpha\rhd(\beta\rightarrow\gamma))}{(\beta\rhd\gamma)}{u}{v} w\rc^ R_{\alpha\rightarrow(\beta\rightarrow\gamma)}\\
&&+\rbsraz{\alpha}{(\beta\rightarrow\gamma)}{u}{\lc v \rc^R_{(\beta\rhd\gamma)}w}+\rbsray{\alpha}{(\beta\leftarrow\gamma)}{u}{v \lc w \rc^S_{\beta\lhd\gamma}}+\rbsraz{\alpha}{(\beta\leftarrow\gamma)}{u}{v \lc w \rc^S_{\beta\lhd\gamma}},
\end{eqnarray*}
}
which is trivial mod $(\mathbb{S}_{\Omega\, S}, w_1)$ if and only if
\allowdisplaybreaks{
\begin{eqnarray*}
(\alpha\rightarrow\beta)\rightarrow\gamma&=&\alpha\rightarrow(\beta\rightarrow\gamma),\\
(\alpha\rightarrow\beta)\rhd\gamma&=&(\alpha\rhd(\beta\rightarrow\gamma))\rightarrow(\beta\rhd\gamma),\\
\alpha\rhd\beta&=&(\alpha\rhd(\beta\rightarrow\gamma))\rhd(\beta\rhd\gamma),\\
(\alpha\rightarrow\beta)\leftarrow\gamma&=&\alpha\rightarrow(\beta\leftarrow\gamma),\\
\alpha\rhd(\beta\leftarrow\gamma)&=&\alpha\rhd\beta,\\
(\alpha\rightarrow\beta)\lhd\gamma&=&\beta\lhd\gamma,\\
(\alpha\leftarrow\beta)\rightarrow\gamma&=&\alpha\rightarrow(\beta\rightarrow\gamma),\\
(\alpha\rhd(\beta\rightarrow\gamma))\leftarrow(\beta\rhd\gamma)&=&(\alpha\leftarrow\beta)\rhd\gamma,\\
(\alpha\rhd(\beta\rightarrow\gamma))\lhd(\beta\rhd\gamma)&=&\alpha\lhd\beta,\\
(\alpha\leftarrow\beta)\leftarrow\gamma&=&\alpha\leftarrow(\beta\rightarrow\gamma),\\
(\alpha\lhd\beta)\rightarrow((\alpha\leftarrow\beta)\lhd\gamma)&=&\alpha\lhd(\beta\rightarrow\gamma),\\
(\alpha\lhd\beta)\rhd((\alpha\leftarrow\beta)\lhd\gamma)&=&\beta\rhd\gamma,\\
(\alpha\leftarrow\beta)\leftarrow\gamma&=&\alpha\leftarrow(\beta\leftarrow\gamma),\\
(\alpha\lhd\beta)\leftarrow((\alpha\leftarrow\beta)\lhd\gamma)&=&\alpha\lhd(\beta\leftarrow\gamma),\\
(\alpha\lhd\beta)\lhd((\alpha\leftarrow\beta)\lhd\gamma)&=&\beta\lhd\gamma.
\end{eqnarray*}
}

For the ambiguities $w_3$, $w_5$, $w_7$ and $w_9$, we write the associated compositions as
$$\Big(\phi^Q_{\gamma,\,\delta}(q|_{\lc u\rc_\alpha^T\lc v\rc_\beta^T}, w), \phi^T_{\alpha,\,\beta}(u, v)\Big)_{w_{Q,\,T}}.$$
Here for $Q,T\in\{R, S\}$ and $\alpha, \beta\in\Omega$,
\begin{align*}
\phi^Q_{\alpha,\,\beta}(u,v):=&\
{\lc u\rc_\alpha^Q\lc v\rc_\beta^Q}-{\lc\lc u\rc_{\alpha\rhd\beta}^R v\rc_{\alpha\rightarrow\beta}^Q}-{\lc u\lc v\rc_{\alpha\lhd\beta}^S\rc_{\alpha\leftarrow\beta}^Q},\\
\phi^T_{\alpha,\,\beta}(u,v):=&\
{\lc u\rc_\alpha^T\lc v\rc_\beta^T}-{\lc\lc u\rc_{\alpha\rhd\beta}^R v\rc_{\alpha\rightarrow\beta}^T}-{\lc u\lc v\rc_{\alpha\lhd\beta}^S\rc_{\alpha\leftarrow\beta}^T}.
\end{align*}
In more details,
\begin{eqnarray*}
\text{including compositions},\quad\qquad&\ \ &\text{ambiguities}\\
\Big(\phi^R_{\gamma,\,\delta}(q|_{\lc u\rc_\alpha^R\lc v\rc_\beta^R}, w), \phi^R_{\alpha,\,\beta}(u, v)\Big)_{w_{R,\, R}}=(f_{\gamma,\,\delta}(q|_{{\rbsla{\alpha}{\beta}{u}{v}}}, w), f_{\alpha,\,\beta}(u, v))_{w_3},&\ \  &w_{R,\,R}=w_3,\\
\Big(\phi^R_{\gamma,\,\delta}(q|_{\lc u\rc_\alpha^S\lc v\rc_\beta^S}, w), \phi^S_{\alpha,\,\beta}(u, v)\Big)_{w_{R, \, S}}=(f_{\gamma,\,\delta}(q|_{{{\rbslb{\alpha}{\beta}{u}{v}}}}, w), g_{\alpha,\,\beta}(u, v))_{w_5},&\ \  &w_{R,\,S}=w_5,\\
\Big(\phi^S_{\gamma,\,\delta}(q|_{\lc u\rc_\alpha^R\lc v\rc_\beta^R}, w), \phi^R_{\alpha,\,\beta}(u, v)\Big)_{w_{S,\,R}}=(g_{\gamma,\,\delta}(q|_{{{\rbsla{\alpha}{\beta}{u}{v}}}}, w), f_{\alpha,\,\beta}(u, v))_{w_7},&\ \  &w_{S,\,R}=w_7,\\
\Big(\phi^S_{\gamma,\,\delta}(q|_{\lc u\rc_\alpha^S\lc v\rc_\beta^S}, w), \phi^S_{\alpha,\,\beta}(u, v)\Big)_{w_{S,\,S}}=(g_{\gamma,\,\delta}(q|_{{\rbslb{\alpha}{\beta}{u}{v}}}, w), g_{\alpha,\,\beta}(u, v))_{w_9},&\ \  &w_{S,\,S}=w_9.
\end{eqnarray*}
Then we get
\allowdisplaybreaks{
\begin{eqnarray*}
&&\Big(\phi^Q_{\gamma,\,\delta}(q|_{\lc u\rc_\alpha^T\lc v\rc_\beta^T}, w), \phi^T_{\alpha,\,\beta}(u, v)\Big)_{w_{Q,\,T}}\\
&=&\phi^Q_{\gamma,\,\delta}(q|_{{\rbslt{\alpha}{\beta}{u}{v}}}, w)-\lc {q|_{\phi^T_{\alpha,\,\beta}(u, v)}}\rc^Q_\gamma\lc{w}\rc^Q_\delta\\
&=&\rbslq{\gamma}{\delta}{q_{\rbslt{\alpha}{\beta}{u}{v}}}{w}-\rbsrqy{\gamma}{\delta}{q_{\rbslt{\alpha}{\beta}{u}{v}}}{w}-\rbsrqz{\gamma}{\delta}{q_{\rbslt{\alpha}{\beta}{u}{v}}}{w}\\
&&-\lc q|_{\rbslt{\alpha}{\beta}{u}{v}-\rbsrty{\alpha}{\beta}{u}{v}-\rbsrtz{\alpha}{\beta}{u}{v}}\rc ^Q_{\gamma} \lc w\ \rc^Q_{\delta}\\
&=&-\rbsrqy{\gamma}{\delta}{q_{\rbslt{\alpha}{\beta}{u}{v}}}{w}-\rbsrqz{\gamma}{\delta}{q_{\rbslt{\alpha}{\beta}{u}{v}}}{w}\\
&&+\lc q|_{\rbsrty{\alpha}{\beta}{u}{v}}\rc^Q_{\gamma} \lc w\rc^Q_{\delta}+\lc q|_{\rbsrtz{\alpha}{\beta}{u}{v}}\rc ^Q_{\gamma}\lc w\rc^Q_{\delta}\\
&\equiv&
-\rbsrqy{\gamma}{\delta}{q|_{\rbsrty{\alpha}{\beta}{u}{v}}}{w}-\rbsrqy{\gamma}{\delta}{q|_{\rbsrtz{\alpha}{\beta}{u}{v}}}{w}\\
&&-\rbsrqz{\gamma}{\delta}{q|_{\rbsrty{\alpha}{\beta}{u}{v}}}{w}-\rbsrqz{\gamma}{\delta}{q|_{\rbsrtz{\alpha}{\beta}{u}{v}}}{w}\\
&&+\rbsrqy{\gamma}{\delta}{q|_{\rbsrty{\alpha}{\beta}{u}{v}}}{w}+\rbsrqz{\gamma}{\delta}{q|_{\rbsrty{\alpha}{\beta}{u}{v}}}{w} \\
&&+\rbsrqy{\gamma}{\delta}{q|_{\rbsrtz{\alpha}{\beta}{u}{v}}}{w}+\rbsrqz{\gamma}{\delta}{q|_{\rbsrtz{\alpha}{\beta}{u}{v}}}{w}\\
&=&0 \quad \mod(\mathbb{S}_{\Omega\, S}, w_{Q,\,T})\,\text{ for }\,Q,T\in\{R, S\}.
\end{eqnarray*}
}
\vskip-0.2in
\mref{it:orbsab} It follows from Theorem~\ref{Composition-Diamond lemma} and Item~\mref{it:orbsaa}.
\end{proof}

In particular, if $S_\omega=R_\omega\,\text{ for }\,\omega\in\Omega$, then an $\Omega$-Rota-Baxter system reduces to an $\Omega$-Rota-Baxter algebra of weight $0$. Free $\Omega$-Rota-Baxter algebras were constructed directly in~\mcite{FP}. Now we give a new method for this free object of weight zero.

\begin{coro}
Let $X$ be a set and $\Omega$ an extended diassociative semigroup. With the order $\leq_\db$ on $\frakM(\Omega, X)$,
\begin{enumerate}
\item  the set
\begin{equation*}
\mathbb{S}_\Omega :=\left\{\rbsla{\alpha}{\beta}{u}{v}-\rbsray{\alpha}{\beta}{u}{v}
-\rbsrac{\alpha}{\beta}{u}{v}\mid u, v\in \frakM(\Omega, X)\,\text{ and }\,\alpha,\beta\in\Omega \right\},
\end{equation*}
is a Gr\"{o}bner-Shirshov basis in $\bfk\frakM(\Omega, X)$.
\mlabel{it:gso}

\item the set
$$\irr(\mathbb{S}_{\Omega}):=\Big\{ w\in \frakM(\Omega, X) \mid  w \neq q|_{\overline{s}}
 \,  \text{ for  any }\, q \in \frakM^{\star}(\Omega, X) \,  \text{ and any }\, s\in \mathbb{S}_{\Omega} \Big\}$$
is a \bfk-basis of the free $\Omega$-Rota-Baxter algebra $\bfk\frakM(\Omega, X)/\Id(\bbs_\Omega)$ of weight zero on $X$.
\mlabel{it:gsob}
\end{enumerate}
\mlabel{coro:forba}
\end{coro}

\begin{proof}
First, we have the following isomorphisms
\begin{align*}
&\ \bfk\frakM(\Omega, X)/\Id(\bbs_\Omega)\\
\cong&\ \bfk\frakM\left(\Omega_R\sqcup \Omega_S, X\right)\big/\Id\left(\bbs_{\Omega }\cup \{ \obs{u}-\obr{u} \mid u\in\frakM(\Omega_R\sqcup \Omega_S, X), \omega\in\Omega\} \right)\\
&\hspace{3cm}\text{(by Remark~\mref{rk:orbs2orb})}\\
\cong& \ \bfk\frakM\left(\Omega_R\sqcup \Omega_S, X\right)\big/\Id(\bbs_{\Omega })\Big/\Id\left(\bbs_{\Omega }\cup \{ \obs{u}-\obr{u} \mid u\in\frakM(\Omega_R\sqcup \Omega_S, X),  \omega\in\Omega\} \right)\big/\Id(\bbs_{\Omega })\\
&\hspace{3cm}\text{(by the third isomorphism theorem)}\\
\cong&\ \bfk\Irr(\bbs_{\Omega })\Big/\Id\big(\{ \obs{u}-\obr{u} \mid u\in\frakM(\Omega_R\sqcup \Omega_S, X), \omega\in\Omega\}\big)\\
&\hspace{3cm}\text{(by Theorem~\ref{thm:gsb})}\\
=&\ \bfk\Big\{ w\in \rswx \mid  w \neq
q|_{\overline{s}}\,  \text{ for  any }\, \ q \in \rswxs, s\in \mathbb{S}_{\Omega }\Big\}\\
 &\ \Big/\Id\big(\{ \obs{u}-\obr{u} \mid u\in\frakM(\Omega_R\sqcup \Omega_S, X), \omega\in\Omega\}\big)\\
\cong&\ \bfk\Big\{w\in\frak{M}(\Omega,\,X)\mid w\neq q|_{\bar{s}}\,\text{ for any }\, q\in\frak{M}^\star(\Omega,\,X), s\in \mathbb{S}_{\Omega}\Big\}\\
=&\ \bfk\Irr(\mathbb{S}_{\Omega}).
\end{align*}
Further, by Theorem~\ref{Composition-Diamond lemma}, $\mathbb{S}_{\Omega}$ is a Gr\"obner-Shirshov basis in $\bfk\frak{M}(\Omega,\, X)$ and so Item~\ref{it:gsob} holds.
\end{proof}

\section{Application of the free $\Omega$-Rota-Baxter system}
\label{sec:application}
In this section, as applications of Theorem~\mref{thm:gsb}, we propose the concepts of Rota-Baxter system family algebras and matching Rota-Baxter systems, and construct their free objects. As examples, free Rota-Baxter systems, free Rota-Baxter family algebras and free matching Rota-Baxter algebras are reconstructed.

\subsection{Gr\"{o}bner-Shirshov bases for Rota-Baxter system family algebras}
\label{sub:family}
The concept of a Rota-Baxter family algebra is a generalization of Rota-Baxter algebra, which plays an important role in quantum renormalization~\mcite{FBP}.

\begin{defn}\mcite{FBP, Guo09}\mlabel{def:pp}
Let $\Omega$ be a semigroup. A pair $(A, (R_\omega)_{\omega\in\Omega})$ consisting of an algebra $A$ and
a collection of linear operators $R_\omega:A\rightarrow A, \omega\in \Omega$
is called a {\bf Rota-Baxter family algebra of weight $\lambda$} if
\begin{equation*}
R_{\alpha}(a)R_{\beta}(b)=R_{\alpha\beta}\left( R_{\alpha}(a)b  + a R_{\beta}(b) + \lambda ab \right),\, \text{ for }\, a, b \in A\,\text{ and }\, \alpha,\, \beta \in \Omega.
\mlabel{eq:RBF}
\end{equation*}
\end{defn}

As a Rota-Baxter system is a generalization of a Rota-Baxter algebra of weight zero, we propose the following concept.

\begin{defn}
Let $\Omega$ be a semigroup. A pair $(A, (R_\omega, S_\omega)_{\omega\in\Omega})$ consisting of an algebra $A$ and two collections of linear operators $R_\omega, S_\omega: A \rightarrow A, \omega\in \Omega$ is called a {\bf Rota-Baxter system family algebra} if, for all $a, b \in A$ and $\alpha,\beta\in\Omega$,
\begin{equation*}
\begin{aligned}
&R_\alpha(a) R_\beta(b)=R_{\alpha\beta}(R_\alpha(a) b+a S_\beta(b)),\\
&S_\alpha(a) S_\beta(b)=S_{\alpha\beta}(R_\alpha(a) b+a S_\beta(b)).
\end{aligned}
\end{equation*}
\mlabel{defn:RS}
\end{defn}

The following is an example of a Rota-Baxter system family algebra.

\begin{exam}\label{exam:RBFsystem}
Let $(A, (R_\omega, S_\omega)_{\omega\in\Omega})$ be an $\Omega$-Rota-Baxter system.
Further, if the four binary operations
\[\leftarrow,\rightarrow,\lhd,\rhd: \Omega\times \Omega \rightarrow \Omega\]
satisfy
$$\alpha\leftarrow\beta=\alpha\rightarrow\beta=:\alpha\cdot\beta\,\text{ and }\, \alpha\lhd\beta=\beta, \alpha\rhd\beta=\alpha.$$
Then $\Omega$ is an extended diassociative semigroup and $(A, (R_\omega, S_\omega)_{\omega\in\Omega})$ reduces to a Rota-Baxter system family algebra.
\end{exam}

Rota-Baxter family algebras are examples of Rota-Baxter system family algebras.

\begin{prop}
Let $\Omega$ be a semigroup and $\lambda\in\bfk.$
The pairs $(A, (R_\omega, R_\omega+\lambda \mathrm{id})_{\omega\in\Omega})$ and $(A, (R_\omega+\lambda \mathrm{id}, R_\omega)_{\omega\in\Omega})$ are Rota-Baxter system family algebras  if and only if
 $(A, (R_\omega)_{\omega\in\Omega})$ is a Rota-Baxter family algebra of weight $\lambda$
\mlabel{prop:fesf}
\end{prop}

\begin{proof}
By symmetry, we only prove the first part.
In terms of Definition~\mref{defn:RS}, the pair $(A, (R_\omega, R_\omega+\lambda \id)_{\omega\in\Omega})$ is a Rota-Baxter system family algebra if and only if
\begin{align*}
 R_\alpha(a)R_\beta(b)=&\ R_{\alpha\beta}(R_\alpha(a)b+aR_\beta(b)+\lambda ab),\\
(R_\alpha+\lambda\id)(a)(R_\beta+\lambda\id)(b)=&\ (R_{\alpha\beta}+\lambda\id)(R_\alpha(a)b+a(R_\beta+\lambda\id)(b)),
\end{align*}
which are equivalent to
$$R_\alpha(a)R_\beta(b)= R_{\alpha\beta}(R_\alpha(a)b+aR_\beta(b)+\lambda ab),$$
as required.
\end{proof}

As an application, we obtain the following result.

\begin{prop}
Let $X$ be a set and $\Omega$ a semigroup.
With the monomial order $\leq_\db$ on $\rswx$,
\begin{enumerate}
\item the set
$$
\mathbb{S}_{SF} :=\left\{
\left.
 \begin{array}{ll}
\lc u\rc^R_\alpha \lc v\rc ^R_\beta- \lc \lc u\rc ^R_\alpha v\rc^R_{\alpha\beta} - \lc u \lc v\rc^S_\beta\rc^R_{\alpha\beta}\\
\lc u\rc^S_\alpha \lc v\rc ^S_\beta- \lc \lc u\rc ^R_\alpha v\rc^S_{\alpha\beta} - \lc u \lc v\rc^S_\beta\rc^S_{\alpha\beta}
\end{array}
\right|u, v\in \rswx\,\text{ and }\,\alpha,\beta\in\Omega \right\}
$$
is a Gr\"{o}bner-Shirshov basis in $\bfk\rswx$. \mlabel{it:gssf}

\item the set
$$\irr(\mathbb{S}_{SF}):=\Big\{ w\in \rswx \mid  w \neq
q|_{\overline{s}}\,  \text{ for  any }\, \ q \in \rswxs\,  \text{ and any }\, s\in \mathbb{S}_{SF}\Big\}$$
is a $\bfk$-basis of the free Rota-Baxter system family algebra $\bfk\rswx/\Id(\mathbb{\mathbb{S}}_{SF})$ on $X$. \mlabel{it:gssf1}
\end{enumerate}
\mlabel{prop:RBF}
\end{prop}

\begin{proof}
The first item follows from Example~\ref{exam:RBFsystem} and Theorem~\mref{thm:gsb}.
The second item is obtained from the first item and Theorem~\ref{Composition-Diamond lemma}.
\end{proof}

As a consequence of Proposition~\ref{prop:RBF}, we obtain a new proof of the following result.

\begin{prop}\cite{QC}
 Let $X$ be a set and $\Omega$ a trivial semigroup with only one element.
With the monomial order $\leq_\db$ on $\rswx$,
\begin{enumerate}
\item the set
\begin{equation*}
\mathbb{S}_S:=\left\{
\left.
 \begin{array}{ll}
\lf u\rf^R\lf v\rf^R-\lf\lf u\rf^R v\rf^R-\lf u\lf v\rf^S\rf^R\\
\lf u\rf^S\lf v\rf^S-\lf \lf u\rf^R v \rf^S-\lf u\lf v\rf^S\rf^S \\
\end{array}
\right|u, v\in \rswx \right\}
\end{equation*}
is a Gr\"{o}bner-Shirshov basis in $\bfk\rswx$. \mlabel{it:orbsbb}

\item the set
$$\Irr(\mathbb{S}_S) := \{ w\in \rswx \mid  w \neq
q|_{\overline{s}}  \,  \text{ for  any }\, \ q \in \rswxs\,  \text{ and any }\, s\in \mathbb{S}_S\}$$
is a $\bfk$-basis of the free Rota-Baxter system $\bfk\rswx/\Id(\mathbb{\mathbb{S}}_S)$. \mlabel{it:orbsbb}
\end{enumerate}
\mlabel{prop:frbs}
\end{prop}

\begin{proof}
It follows from Proposition~~\ref{prop:RBF} by taking $\Omega$ to be a trivial semigroup.
\end{proof}

If $S_\omega=R_\omega+\lambda\id\,\text{ for }\,\omega\in\Omega$, then a Rota-Baxter system family algebra reduces to a Rota-Baxter family algebra by Proposition~\mref{prop:fesf}.

\begin{prop}\cite{ZG}\label{prop:gsbases}
Let $X$ be a set and $\Omega$ a semigroup. With the order $\leq_\db$ on $\frakM(\Omega, X)$,
\begin{enumerate}
\item  the set
$$\mathbb{S}_F :=\left\{\lc u\rc_\alpha \lc v\rc _\beta- \lc \lc u\rc _\alpha v\rc_{\alpha\beta} - \lc u \lc v\rc_\beta\rc_{\alpha\beta}- \lambda\lc u  v\rc_{\alpha\beta}\mid  u, v\in \frakM(\Omega, X)\,\text{ and }\,\alpha, \beta\in \Omega \right\}.$$
is a Gr\"{o}bner-Shirshov basis in $\bfk\frakM(\Omega, X)$.
\mlabel{it:gsf}

\item the set
$$\irr(\mathbb{S}_{F}):=\Big\{ w\in \frakM(\Omega, X) \mid  w \neq q|_{\overline{s}}
\,  \text{ for  any }\, \ q \in \frakM^{\star}(\Omega, X)\,  \text{ and any }\, s\in \mathbb{S}_{F} \Big\}$$
is a \bfk-basis of the free Rota-Baxter family algebra $\bfk\frakM(\Omega, X)/\Id(\bbs_F)$ on $X$.
\mlabel{it:gsfb}
\end{enumerate}
\label{prop:rbf}
\end{prop}

\begin{proof}
First, we obtain
\begin{align*}
&\ \bfk\frakM(\Omega, X)/\Id(\bbs_F)\\
\cong& \ \bfk\frakM\left(\Omega_R\sqcup \Omega_S, X\right)\big/\Id\left(\bbs_{F}\cup \{ \obs{u}-\obr{u} -\lambda u \mid u\in\frakM(\Omega_R\sqcup \Omega_S, X), \omega\in\Omega\} \right)\\
&\hspace{3cm}\text{(by Proposition~\mref{prop:fesf})}\\
\cong&\  \bfk\frakM\left(\Omega_R\sqcup \Omega_S, X\right)\big/\Id(\bbs_{F})\Big/\Id\left(\bbs_{F}\cup \{ \obs{u}-\obr{u} -\lambda u \mid u\in\frakM(\Omega_R\sqcup \Omega_S, X),\omega\in\Omega\} \right)\big/\Id(\bbs_{F})\\
&\hspace{3cm}\text{(by the third isomorphism theorem)}\\
\cong&\ \bfk\Irr(\bbs_{F})\Big/\Id\big(\{ \obs{u}-\obr{u} -\lambda u \mid u\in\frakM(\Omega_R\sqcup \Omega_S, X), \omega\in\Omega\}\big)\\
&\hspace{3cm}\text{(by Proposition~\ref{prop:RBF})}\\
=&\ \bfk\Big\{ w\in \rswx \mid  w \neq
q|_{\overline{s}}\,  \text{ for  any }\, \ q \in \rswxs, s\in \mathbb{S}_{F}\Big\}\\
&\ \Big/\Id\big(\{ \obs{u}-\obr{u} -\lambda u \mid u\in\frakM(\Omega_R\sqcup \Omega_S, X),\omega\in\Omega\}\big)\\
\cong&\ \bfk\Big\{w\in\frak{M}(\Omega,\,X)\mid w\neq q|_{\bar{s}}\,\text{ for any }\, q\in\frak{M}^\star(\Omega,\,X),s\in \mathbb{S}_{F}\Big\}\\
=&\ \bfk\Irr(\mathbb{S}_{F}).
\end{align*}
Further, by Theorem~\ref{Composition-Diamond lemma}, $\mathbb{S}_{F}$ is a Gr\"obner-Shirshov basis in $\bfk\frak{M}(\Omega,\, X)$ and hence Item~\mref{it:gsfb} holds.
\end{proof}

\subsection{Gr\"{o}bner-Shirshov bases for matching Rota-Baxter systems}
This subsection is devoted to supply Gr\"{o}bner-Shirshov bases for \mrbss.
Let us first review the concept of matching Rota-Baxter algebras.

\begin{defn}\cite{ZGG}
Let $\Omega$ be a set and $(\lambda_\omega)_{\omega\in\Omega}$ be a collection of elements in $\bfk$.
A pair $(A, (R_\omega)_{\omega\in\Omega})$ consisting of an algebra $A$ and a collection of linear operators
$R_\omega: A \rightarrow A, \omega\in \Omega$ is called a {\bf matching Rota-Baxter algebra of  weight $(\lambda_\omega)_{\omega\in\Omega}$} if, for all $a, b \in A$, $\alpha,\beta\in \Omega$,
$$R_\alpha(a)R_\beta(b)=R_\beta(R_\alpha(a)b)+R_\alpha(aR_\beta(b))+\lambda_\beta R_\alpha(ab).$$
\end{defn}

Combining the above concept and Definition~\ref{defn:mrba}, we propose

\begin{defn}\mlabel{defn:matching}
Let $\Omega$ be a set.
A pair $(A, (R_\omega, S_\omega)_{\omega\in\Omega})$ consisting of an algebra $A$ and two collections of linear operators
$R_\omega, S_\omega: A \rightarrow A, \omega\in \Omega$ is called a {\bf \mrbs} if, for all $a, b \in A$, $\alpha,\beta\in \Omega$,
\begin{align*}
&R_\alpha(a)R_\beta(b)=R_\beta(R_\alpha(a)b)+R_\alpha(aS_\beta(b)), \\
&S_\alpha(a)S_\beta(b)=S_\beta(R_\alpha(a)b)+S_\alpha(aS_\beta(b)).
\end{align*}
\end{defn}

\begin{remark}
In the above concept, if $S_\omega=R_{\omega}+\lambda_\omega\id\,\text{ for }\,\omega\in\Omega$, then we recover the notation of a matching Rota-Baxter algebra of weight $(\lambda_\omega)_{\omega\in \Omega}$.
\end{remark}

The following is an example of a matching Rota-Baxter system.

\begin{exam}\label{exam:mRBsystem}
Let $(A, (R_\omega, S_\omega)_{\omega\in\Omega})$ be an $\Omega$-Rota-Baxter system.
Further, if the four binary operations
\[\leftarrow,\rightarrow,\lhd,\rhd: \Omega\times \Omega \rightarrow \Omega\]
satisfy
$$\alpha\rightarrow\beta=\alpha\lhd\beta=\beta\,\text{ and }\, \alpha\leftarrow\beta=\alpha\rhd\beta=\alpha.$$
Then  $\Omega$ is an extended diassociative semigroup and $(A, (R_\omega, S_\omega)_{\omega\in\Omega})$ reduces to a matching Rota-Baxter system.
\end{exam}

A matching Rota-Baxter algebra is a special case of a matching Rota-Baxter system.

\begin{prop}
Let $\Omega$ be a set and $(\lambda_\omega)_{\omega\in\Omega}$ be a collection of elements in $\bfk$.
The pairs $(A, (R_\omega, R_\omega+\lambda_\omega \mathrm{id})_{\omega\in\Omega})$ and $(A, (R_\omega+\lambda_\omega \mathrm{id}, R_\omega)_{\omega\in\Omega})$ are matching Rota-Baxter systems if and only if
$(A, (R_\omega)_{\omega\in\Omega})$ is a matching Rota-Baxter algebra of weight $(\lambda_\omega)_{\omega\in\Omega}$.
\mlabel{prop:mesm}
\end{prop}

\begin{proof}
We just show the first one, as the second one is similar.
By Definition~\mref{defn:matching}, the pair $(A, (R_\omega, R_\omega+\lambda_\omega \mathrm{id})_{\omega\in\Omega})$ is a matching Rota-Baxter system if and only if
\begin{align*}
R_\alpha(a)R_\beta(b)=&\ R_{\beta}(R_\alpha(a)b)+R_{\alpha}(aR_\beta(b)+\lambda_\beta a b),\\
(R_\alpha(a)+\lambda_\alpha a)(R_\beta(b)+\lambda_\beta b)=&\  (R_{\beta}+\lambda_\beta \id)(R_\alpha(a)b)+(R_{\alpha}+\lambda_{\alpha} \id)(a(R_\beta+\lambda_\beta\id)(b)),
\end{align*}
which are equivalent to
$$R_\alpha(a)R_\beta(b)= R_{\beta}(R_\alpha(a)b)+R_{\alpha}(aR_\beta(b)+\lambda_\beta a b).$$
This shows that $(A, (R_\omega)_{\omega\in\Omega})$ is a matching Rota-Baxter algebra of weight $(\lambda_\omega)_{\omega\in\Omega}$.
\end{proof}

As a direct consequence, we have

\begin{prop}
Let $X$ be a set and $\Omega$ a nonempty set.
With the monomial order $\leq_{_{{\rm db}}}$ on $\rswx$,
\begin{enumerate}
\item
the set
$$
\mathbb{S}_{MS} :=\left\{
\left.
 \begin{array}{ll}
\lc u\rc ^ R_\alpha \lc v \rc^R_\beta-\lc \lc u\rc ^R_\alpha v\rc^R_\beta-\lc u\lc v\rc ^S_\beta\rc^R_\alpha\\
\lc u\rc ^ S_\alpha \lc v \rc^S_\beta-\lc \lc u\rc ^R_\alpha v\rc^S_\beta-\lc u\lc v\rc ^S_\beta\rc^S_\alpha \\
\end{array}
\right|u, v\in \rswx\,\text{ and }\,\alpha,\beta\in\Omega \right\}
$$
is a Gr\"{o}bner-Shirshov basis in $\bfk\rswx$.

\item the set $$
\irr(\mathbb{S}_{MS}) :=\Big\{ w\in \rswx \mid  w \neq
q|_{\overline{s}}
\,  \text{ for  any }\, q \in \rswxs\,  \text{ and any }\, s\in \mathbb{S}_{MS}\Big\}$$
is a $\bfk$-basis of the free \mrbs $\bfk\rswx/\Id(\mathbb{\mathbb{S}}_{MS})$.
\end{enumerate}
\mlabel{prop:fmrb}
\end{prop}

\begin{proof}
The first one is obtained by Example~\ref{exam:mRBsystem} and Theorem~\mref{thm:gsb}.
Further, the second one is valid by Theorem~\ref{Composition-Diamond lemma}.
\end{proof}

In particular, if $S_\omega=R_\omega+\lambda_\omega\id\,\text{ for }\,\omega\in\Omega$, we obtain the Gr\"obner-bases for matching Rota-Baxter algebras as follows.

\begin{prop}~\label{prop:fmrba}
 Let $X$ be a set and $\Omega$ a nonempty set. With the monomial order $\leq_\db$ on $\frakM(\Omega, X)$,
\begin{enumerate}
\item\label{it:gsm}  the set
$$
\mathbb{S}_M=\left\{
\left.
 \begin{array}{ll}
\lc u\rc_\alpha \lc v \rc_\beta-\lc \lc u\rc _\alpha v\rc_\beta-\lc u\lc v\rc _\beta\rc_\alpha-\lambda_\beta\lc uv \rc_\alpha \\
\end{array}
\right|u, v\in \frakM(\Omega, X)\,\text{ and }\,\alpha,\beta\in\Omega \right\}
$$
is a Gr\"{o}bner-Shirshov basis in $\bfk\frakM(\Omega, X)$.

\item \label{it:gsmb} the set
$$\irr(\mathbb{S}_{M}):=\Big\{ w\in \frakM(\Omega, X) \mid  w \neq q|_{\overline{s}}
\,  \text{ for  any } \, q \in \frakM^{\star}(\Omega, X)\,  \text{ and any }\, s\in \mathbb{S}_{M} \Big\}$$
is a \bfk-basis of the free matching Rota-Baxter algebra $\bfk\frakM(\Omega, X)/\Id(\bbs_M)$ on $X$.
\end{enumerate}
\end{prop}

\begin{proof}
The proof is similar to the one of Proposition~\ref{prop:rbf}.
\end{proof}

\noindent
{{\bf Acknowledgments.} This work is supported in part by Natural Science Foundation of China (No. 12070091, 12101183), project funded by China
Postdoctoral Science Foundation (No. 2021M690049) and the Natural Science Project of Shaanxi Province (No. 2022JQ-035).

\end{document}